\documentclass[a4paper,12pt]{amsart}

\newtheorem{thm}{Theorem}[section]
\newtheorem{cor}[thm]{Corollary}
\newtheorem{lem}[thm]{Lemma}

\theoremstyle{definition}
\newtheorem{exam}[thm]{Example}
\newtheorem{defn}[thm]{Definition}
\theoremstyle{remark}
\newtheorem{rem}[thm]{Remark}
\numberwithin{equation}{section}

\begin{document}
\title[]{\textsc{\textbf{Conditional Carleson measures and related operators on Bergman spaces}}}
\author{\textsc{A. Aliyan, Y. Estaremi and A. Ebadian}}
\address{\textsc{A. Aliyan}} \email{aliyan.amir1404@gmail.com}
\address{\textsc{Y. estaremi}} \email{estaremi@gmail.com}
\address{\textsc{A. Ebadian}} \email{ebadian.ali@gmail.com}

\address{Department of mathematics, Payame Noor university , P. O. Box: 19395-3697, Tehran,
Iran\\
}

\thanks{}
\thanks{}
\subjclass[2010]{}
\date{}
\dedicatory{}

\begin{abstract}
In this paper first we define generalized Carleson measure. Then we consider a special case of it, named conditional Carleson measure on the Bergman spaces. After that we give a characterization of conditional Carleson measures on Bergman spaces. Moreover, by using this characterization we find an equivalent condition to boundedness of weighted conditional expectation operators on Bergman spaces.
\end{abstract}
\maketitle
\commby{}

\section{Introduction and Preliminaries}
Let $\mathbb{C}$  be the complex plane,
$\mathbb{D} = \{ z \in \mathbb{C}:\left| z \right| < 1\}$ and $\mathbb{T} = \{ z \in \mathbb{C}:\left| z \right| = 1\}$. Likewise, we write $\mathbb{R}$ for the real line. The normalized
area measure on $\mathbb{D}$ will be denoted by $dA$. As is known $dA(z) = \dfrac{1}{\pi}dxdy = \dfrac{1}{\pi}rdrd\theta$, in which $z = x + iy = r{e^{i\theta }}$.
For $0< p< +\infty $ and $-1< \alpha< +\infty $, the (weighted) Bergman space
$L_{a }^{p,\alpha}={L}_{a }^{p,\alpha}(\mathbb{D})$ of $\mathbb{D}$, is the space of all analytic functions in $ {{L}^{p}}(\mathbb{D},d{{A}_{\alpha }})$,  where
 $$d{{A}_{\alpha }}(z)=(\alpha +1){{(1-{{\left| z \right|}^{2}})}^{\alpha }}d{{A}}(z).$$
Let $f\in{{L}^{p}}(\mathbb{{D}},d{{A}_{\alpha }})$ and
$${{\left\| f \right\|}_{p.\alpha }}={{\left[ \int_{\mathbb{D}}{{{\left| f(z) \right|}^{p}}d{{A}_{\alpha }}(z)} \right]}^{{}^{1}/{}_{p}}}.$$
Then for $1\leq p<\infty$, the space ${{L}^{p}}(\mathbb{D},d{{A}_{\alpha }})$ is a Banach space with respect to the norm ${{\left\| . \right\|}_{p.\alpha }}$. In addition, for $0< p<1$, the space ${{L}^{p}}(\mathbb{D},d{{A}_{\alpha }})$ is a complete metric space with respect to the metric defined by
$d\left( f,g \right)=\left\| f-g \right\|_{p.\alpha }^{p}$.\\

Recall that for $1 < p < \infty$ and $ -1<\alpha ,$ the projection $P_{\alpha}$ from  $ L^{\alpha,p}$  into $ L_a^{\alpha,p}$ is given by
$$P_\alpha(f)(w)= \mathop{\int}_{\mathbb{D}} \frac{f(z)}{(1- w \overline z)^{2+\alpha}}dA_\alpha(z) , \quad w \in {\mathbb{D}}.$$
The projection $P_\alpha$ is called the (weighted) Bergman projection on $\mathbb{D}$. Let $\mu$ be a finite positive Borel measure on $\mathbb{D}$ and $0<p<\infty$. We say that $\mu$ is a Carleson measure on the Bergman space $L_a^p$, if there exists a  constant $C > 0$ such that
\[{\int_{\mathbb{D}} {\left| {f(z)} \right|} ^p}d\mu (z) \le C{\int_{\mathbb{D}} {\left| {f(z)} \right|} ^p}dA(z),\]
for all $f\in L_a^p$.
In addition, the functions  $$ K_\alpha(\omega,z)= \frac{1}{(1- w \overline z)^{2+\alpha}}\qquad z,\omega\in\mathbb{D}$$
are called the (weighted) Bergman kernels of $ \mathbb{D}$.\\

Let $(\mathcal{X},\mathcal{M},\mu)$  be a complete sigma-finite measure space and let $\mathcal{A}$ be a sigma-subalgebra of $\mathcal{M}$ such that $(\mathcal{X},\mathcal{A},\mu\mid_\mathcal{A})$
is also sigma-finite.The collection of (equivalence classes modulo sets of zero measure) $\mathcal{M}$-measurable complex-valued
functions on $\mathcal{X}$ will be denoted $L^0(\mathcal{M})$.
 Moreover, we let
 $L^p(\mathcal{M}) = L^p(\mathcal {X},\mathcal{ M},\mu)$ and
$ L^p(\mathcal{A}) = L^p(\mathcal{X},\mathcal{A},\mu\vert_\mathcal{A})$, for $1\leq p<\infty$. As a consequence of the Radon-Nikodym theorem we have that for each non-negative function $f \in L^0(\mathcal{M})$
there exists a unique non-negative $\mathcal{A}$-measurable function  $E_{\mathcal{A}}( f )\in L^0(\mathcal{A})$ such that
$\int\limits_{\Delta }{E(f)d\mu}=\int\limits_{\Delta}{f d\mu }$
for all $ \Delta ~\in\mathcal{A}$ . The function $ E_{\mathcal{A}}( f )$ is called the conditional expectation of $f$ with respect to $\mathcal{A}$.

Note that if $E _{\mathcal{A}} P_{\alpha} = P_{\alpha} E_{\mathcal{A}}$ on $L^p(\mathbb{D}, \mathcal{M}, \mathcal{A})$, then $L_a^p(\mathbb{D})=L_a^{p,0}(\mathbb{D})$ is invariant under the conditional expectation operator $E_{\mathcal{A}}$, i.e.,  $E_{\mathcal{A}}(L_{a}^{p}(\mathbb{D})) \subseteq L_{a}^{p}(\mathbb{D})$.
Suppose that $\mathcal{M}$ is the sigma-algebra of  the Lebesgue  measurable sets in $\mathbb{D}$, $\mathcal{ A}$ is a subalgebra of $\mathcal{M}$  and $E=E_{\mathcal{A}}$  is the related conditional expectation operator.
 For a non-constant analytic self-map $\varphi$ on $\mathbb{D}$ and $ z\in \mathbb{D}$,
put $c_z =\{\zeta \in \mathbb{D}_0 : \varphi(\zeta) = \varphi(z)\}$, where $\mathbb{D}_{0} = \{\zeta \in \mathbb{D} :\varphi '(z)(\zeta)\neq 0\}$.\\
The map $\varphi$ has finite multiplicity if we can find some $n \in \mathbb{N}$ such that for each $ z \in \mathbb{D}$, the
level set $c _{z}$ consists of at most $n$ points.
 By $\mathcal{A}= \mathcal{A}(\varphi)$ we denote the sigma-algebra generated by $\{ \varphi^{-1}(U): U \subset \mathbb{C}~ is ~open \}$ .
  For the Lebesgue measure $m$ on $\mathbb{C}$ we get that $h = \frac{{d{A_\circ }{\varphi ^{ - 1}}}}{{dm}}$ is almost everywhere finite valued, because the finite measure $A \circ \varphi^{-1}$ is absolutely continuous with respect to $m$. Now we define the weighted conditional expectation operator on $L_{\alpha }^{p.\alpha }$ as follow:
\begin{defn} Let $u$ be an analytic function on $\mathbb{D}$,  $0<p<\infty$ and   $\alpha ,\beta >-1$. If $\mathcal{A}$ is a sigma-subalgebra of the sigma-algebra of the Lebesgue measurable sets, then the weighted conditional expectation operator  $T=M_{u}E_{\mathcal{A}}$ on $L_{\alpha }^{p.\alpha }$ is defined as
  $T(f)(z)=u(z)E_{\mathcal{A}}(f)(z)$, for each $f\in L_{\alpha }^{p.\alpha }$ and $z\in \mathbb{D}$ such that $E_{\mathcal{A}}(f)$ is an analytic function.
\end{defn}

Joseph Ball in \cite{jb} studied conditional expectation on $H^p$ and he find that $H^p$ is invariant under $E_{\mathcal{A}(f)}$ for an inner function $f$. After that Aleksandrov  in \cite{al} proved that for a complete sigma-subalgebra $\mathcal{A}$ of $\mathcal{M}$, the sigma-algebra of Lebesgue measurable subsets of $\mathbb{T}$. Then $\mathcal{A}=\mathcal{A}(f)$ for some inner function $f$ if and only if $PE_{\mathcal{A}}=E_{\mathcal{A}}P$, in which $P$ is the Riesz projection, i.e. the orthogonal projection from $L^2$ onto the Hardy space $H^2$. If we compose the discussions presented in \cite{hjj}, \cite{cs} and \cite{jh}, then we will get a formula for conditional expectation operator corresponding to $\mathcal{A}(\varphi)$, in which $\varphi$ is an arbitrary non-constant analytic self-map on $\mathbb{D}$.\\
In \cite{jm}, the authors investigated some basic properties of some conditional expectation-type Toeplitz operators on Bergman spaces. In this paper we consider the weighted conditional expectation operators on the Bergman spaces and we will give some results based on Carleson measures. Here we recall some basic Lemmas that we need them in the sequel.
\begin{lem}\cite{5}\label{3}  If $0<p <\infty$  and  $\alpha >-1$. Then
$$|f(z)|\leq \frac{{\Vert f \Vert}_{p,\alpha }}{{(1-\vert z |^{2})}^{(2+\alpha)/p}}$$
for all  $f\in L_a^{p,\alpha}$ and $z\in \mathbb{D}$.
\end{lem}
\begin{lem}\cite{27}\label{3.5} Let $0<p <\infty$ and $\alpha>-1$. Then for each $z\in \mathbb{D}$ we have
$$\mathop {\sup }\limits \{ \left| {f(z)} \right|:f \in L_a^{p,\alpha },{\left\| f \right\|_{L_a^{p,\alpha }}}\leq 1\}= \frac{1}{{{{(1 - |z{|^2})}^{(2 + \alpha )/p}}}}.$$
\end{lem}
 Let $a \in\mathbb{D}$ and $$\varphi_{a}(z)= \frac{a-z}{1- \overline{a} z}, \quad z \in \mathbb{D}.$$
 The analytic transformation $\varphi_{a}$ is called the involutive Mobius transformation of $\mathbb{D}$, which interchanges the origin and $z$. It is clear  for all $z\in \mathbb{D}$,${\varphi _a}({\varphi _a}(z)) = z$ or ${\varphi _a}^{ - 1}={\varphi _a}$ and
$${{\varphi'}_a}(z)=-\frac{{1-{{\left|a \right|}^2}}}{{{{(1-\bar az)}^2}}}.$$ \\
The metric defined on $\mathbb{D}$ by
 $$\rho(a, z)=|\varphi_a(z)|= \big|\frac{a-z}{1- \overline a z} \big|, \quad \quad a,z \in \mathbb{D},$$
 is called pseudo-hyperbolic distance. Maybe one of important properties of the pseudo-hyperbolic distance is
that it is Mobius invariant, that is $$\rho(w, z)=\rho(\varphi_a(w),\varphi_a(z)).$$
In addition the metric given by
$$\beta(a,z) = \frac{1}{2} log \frac{1+ \rho(a,z)}{1- \rho(a, z)}, \quad \quad  a, z \in \mathbb{D},$$
is called Bergman metric on $\mathbb{D}$( or some where is called the hyperbolic metric or the Poincar\'e metric on $\mathbb{D}$).
Hence we have $$\beta(0,z) = \frac{1}{2} log \frac{1+ {\left| z \right|}}{1- {\left| z \right|}}, \quad \quad   z \in \mathbb{D}.$$

The Bergman metric is also Mobius invariant. This means that
$$\beta(w, z)=\beta(\varphi_a(w),\varphi_a(z)).$$
for all $\varphi \in Aut(\mathbb{D})$ and all $z,w \in \mathbb{D}.$

For any $a \in \mathbb{D}$ and $r > 0$, let
$$D(a,r) = \lbrace z \in \mathbb{D}:~ \beta(a, z) < r\rbrace$$
be the Bergman metric disk with "center" $a$  and "radius" $r$. It follows from the expression for $\beta(a,z)$  that $D(a,r)$ is a Euclidean disk with Euclidean center and radius
$$C= \frac{1-s^{2}}{1- s^{2}|a|^{2} }a, \quad \quad R= \frac{1-|a|^{2}}{1- s^{2} |a|^{2}}s,$$
where  $s=\tanh r$.We have $D(a,r)=\varphi_a(D(0,r))$.
As is known for each $z \in \mathbb{D}$, the Riesz representation theorem implies that there exists a unique function $K_z$ in $L_a^2(\mathbb{D})$ such that for all $f\in L_a^2(\mathbb{D})$
\[f(z) = \int_{\mathbb{D}} {f(w)\overline {{K_z}(w)} dA(w)}.\]  Let $K(z,a)$ be the function on $\mathbb{D}\times \mathbb{D}$ defined by $K(z,a)=\overline {{K_z}(a)}$. $K(z,a)$ is called the Bergman kernel of $\mathbb{D}$ or the reproducing kernel of $L_a^2(\mathbb{D})$ because of the formula
$$f(z) = \int_{\mathbb{D}} f(w) K(z,w) dA(w)$$
reproduces each $f\in L_a^2(\mathbb{D})$. And it is known that $K(z,w)=\frac{1}{{{{(1 - z\overline a )}^2}}}$. Moreover the functions  $$ K_\alpha(\omega,z)= \frac{1}{(1- w \overline z)^{2+\alpha}}\qquad z,\omega\in\mathbb{D}$$
are called (weighted) Bergman kernels of $\mathbb{D}$. If we set
\[{k_a}(z) = \frac{{K(z,a)}}{{\sqrt {K(a,a)} }} = \frac{{1 - {{\left| a \right|}^2}}}{{{{(1 - \bar az)}^2}}},\]
then $k_a(z)\in L_a^2(\mathbb{D})$ and they are called normalized reproducing kernels of $L_2(\mathbb{D})$. Easily we get that the derivative of $\varphi_a$ at $z$ is equal to $—k_a(z)$. This implies that
 $$\left|{{{\varphi'}_a}(z)}\right|=\left|{k_a(z)}\right|$$
for all $a,z\in\mathbb{D}$.

\begin{rem}\label{4} Let $0<p <\infty$, $\alpha >-1$, and $a\in\mathbb{D}$. The  normalized Bergman kernel function for $L_{a}^{p,\alpha}$ is defined by
$$({k_{a}(z)})^{\frac{\alpha+2}{p}}=f_{a}^{p,\alpha}(z)=\frac{\big(1-|a|^{2}\big)^{(2+ \alpha) /p}}{( 1- \overline{a} z )^{2(2+\alpha)/p}}.$$
 Direct computations shows that
 $$ \left\| {f_a^{p,\alpha}} \right\|_{L_a^{p,\alpha }}^p={(1 - {\left| a \right|^2})^{2 + \alpha }}\left\| {{{(1 - \bar az)}^{ - 2 - \alpha }}} \right\|_{L_a^{\alpha ,2}}^2.$$
 Therefore for all $\alpha > −1$ and $ p > 0$ we have
$||f_{a}^{p,\alpha}||_{L_{a}^{p,\alpha}}^{p}=1$.
Let $f\in L_a^{p,\alpha}$ and $a\in \mathbb{D}$. If we set
$$F(z)=f(\varphi_a(z))f_{a}^{p,\alpha}(z)\qquad z\in \mathbb{D},$$
then  ${\left\| F \right\|_{L_a^{p,\alpha }}}={\left\| f \right\|_{L_a^{p,\alpha }}}$.
\end{rem}
Here we recall some results of \cite{11} that are basic and useful for in the sequel.
\begin{lem}\label{4.5}\cite{11} For each $r > 0$ there exists a positive constant $C_r$ such that
$$C_r^{ - 1} \le \frac{{1 - {{\left| a \right|}^2}}}{{1 - {{\left| z \right|}^2}}} \le {C_r}$$
and
$$C_r^{ - 1} \le \frac{{1 - {{\left| a \right|}^2}}}{{\left| {1 - \bar az} \right|}} \le {C_r},$$
for all a and z in $\mathbb{D}$ with $\beta(a, z) < r$. Moreover, if $r$ is bounded above, then we may
choose $C_r$ to be independent of $r$.
\end{lem}
By some elementary calculations we have
$$\frac{(1-|a|^{2})^{2}}{|1- \overline{a}z|^{4}} \approx \frac{1}{(1-|z|^{2})^{2}} \approx \frac{1}{(1-|a|^{2})^{2}} \approx \frac{1}{|D(a,r)|},$$
for $\beta(a,z) \leq R$ and
$$|{D} (z,r)|_{A} \approx (1-|z|^{2})^{2} \approx (1- |\omega|^{2})^{2} \approx|{D} (\omega,s)|_{A}$$
for $\beta(z,\omega) \leq R$.\\
Note that for all $z, w\in \mathbb{D}$, we have \[\frac{1}{{1 - \left| w \right|}}\leq\frac{1}{{\left| {1 - z\overline w } \right|}} \le \frac{1}{{1 - \left| z \right|}} = \frac{{1 + \left| z \right|}}{{1 - {{\left| z \right|}^2}}} \le \frac{2}{{1 - {{\left| z \right|}^2}}}.\]
 The notation $A \approx B$ means that there is a positive constant $C$ independent of $A$ and $B$ such that $C^{-1}B \leq A \leq CB$.\\

Now by using the results of \cite{11} we obtain that there is a cover of disjoint balls in $\mathcal{A}$ for $D$.
\begin{lem}\label{8} There is a positive integer $N$ such that for any $r \leq 1$, there exists a sequence $\{ a_{n}\}$ in $\mathbb{D}$ such that $ D(a_{n}, r)\in \mathcal{A}$ and satisfying the following conditions:\\
(1)$\mathbb{D}= \bigcup_{n=1}^{+\infty} D(a_{n}, r)$\\
(2) $D(a_{n}, \frac{r}{4}) \cap D(a_{m},\frac{r}{4})= \emptyset$ if $n \neq m$;\\
$(3)$Any point in $\mathbb{D}$ belongs to at most $N$ of the sets $D(a_{n}, 2r)$.
\end{lem}

Here we recall a theorem concerning of the form of the functions in the range of conditional expectation operators.
\begin{thm}\cite{jh}\label{12} Suppose that $\mathcal{A}= \mathcal{A}(\varphi)$ for some $\varphi \in A(\mathbb{D})$ with finite multiplicity. Suppose that none of the $\zeta_{j}(w)$ belongs to $\{z:\varphi '(z) = 0\} $ and that $w \notin f(\mathbb{T})$. Then for every $f$ in $L_{a}^{p}(\mathbb{D})$ and $\zeta$ in $\varphi^{-1}(w)$,
\begin{flushleft}
$$E_{\mathcal{A}}(f)(\zeta)=\frac{\sum_{\zeta_j \epsilon c_z} \frac{f(\zeta_{j})}{|{\varphi }'({{\xi }_{j}})|^{2}}}{\sum_{\zeta_j \epsilon c_z}\frac{1}{|{\varphi }'({{\xi }_{j}})|^{2}}}.$$
\end{flushleft}
Also, the function $\omega$ defined as
$$\omega(\zeta)= \frac{\frac{1}{|{\varphi }'({{\zeta }_{j}})|^{2}}}{\sum_{\zeta_j \epsilon c_z}\frac{1}{|{\varphi }'({{\zeta }_{j}})|^{2}}}$$
is constant on each level set. In particular if $E_\mathcal{A}P_\alpha=P_\alpha E_\mathcal{A}$ , then $\omega$ is constant on $\mathbb{D}$.
\end{thm}
Every non-constant analytic self-map $\varphi$ on $\mathbb{D}$ has finite multiplicity, since if there is $w\in \varphi(\mathbb{D})$ such that $\varphi^{-1}(\{w\})$ is infinite, then $\varphi(z)=w$ for all $z\in \mathbb{D}$. Therefore the Theorem \ref{12} holds for all non-constant analytic self-map $\varphi$ on $\mathbb{D}$.
Now we recall some assertions that we will use them in the sequel.
\begin{lem}\cite{14}\label{15.5}
There is a constant $C>0$ such that
$$|f(z)|^{p} \leq \frac{C}{{(1-|z|^2)}^{\alpha+2}} \int_{D(z,r)} |f(\omega)|^{p} dA_\alpha(\omega),$$
for all $f$ analytic, $z \in \mathbb{D}, p>0$, and $r \leq 1$.
We especially note that the constant $C$ above is independent of $r$ and
$p$. The restriction $r < 1$ above can be replaced by $r < R$ for any positive
number $R.$
\end{lem}
\begin{lem}\cite{14}\label{3.6} For any $ r>0$ and $a\in D$, we have the following equalities:
$$\left| {D(a,r)} \right| = \frac{{{{(1 - {{\left| a \right|}^2})}^2}{s^2}}}{{{{(1 - {{\left| a \right|}^2}{s^2})}^2}}},$$
$$\mathop {\inf }\limits_{z \in D(a,r)} {\left| {k_a(z)} \right|^2} = \frac{{{{(1 - s\left| a \right|)}^4}}}{{{{(1 - {{\left| a \right|}^2})}^2}}},$$
$$\mathop {\sup }\limits_{z \in D(a,r)}{\left| {k_a(z)} \right|^2} = \frac{{{{(1 + s\left| a \right|)}^4}}}{{{{(1 - {{\left| a \right|}^2})}^2}}},$$
where $s=tanh\, r\in(0,./762)$ for $r\in(0,1)$ and $\left| {D(a,r)} \right|$ is the (normalized) area of D(a,r).
\end{lem}
\begin{lem}\label{16} Let $0<p\leq q <\infty$ and $z\in D$. Then for each $f\in L_a^{p,\alpha}(\mathbb{D})$ we have
$$\left( {\int_{D} {{{\left| {f(z)} \right|}^q}d\mu(z) } } \right) \leq{({\left\| f \right\|_{L_a^{p,\alpha }}})^q}{\left( {\mu (D)} \right)^{\frac{{p - q}}{p}}}.$$
\begin{proof}
 \begin{eqnarray}
\nonumber \left( {\int {{{\left| {f(z)} \right|}^q}d\mu(z) } } \right) &\le & {\left( {\int {{{\left( {{{\left| {f(z)} \right|}^q}} \right)}^{\frac{p}{q}}}d\mu(z) } } \right)^{\frac{q}{p}}}{\left( {\int {d\mu(z) } } \right)^{\frac{{p - q}}{p}}} \\
\nonumber &=& {\left( {\int {{{\left( {{{\left| {f(z)} \right|}^q}} \right)}^{\frac{p}{q}}}d\mu(z) } } \right)^{\frac{q}{p}}}{\left( {\mu(D)} \right)^{\frac{{p - q}}{p}}}\\
\nonumber &=& {({\left\| f \right\|_{L_a^{p,\alpha }}})^q}{\left( {\mu (D)} \right)^{\frac{{p - q}}{p}}}.
\end{eqnarray}
\end{proof}
\end{lem}
Hence by this lemma we conclude that $\int\limits_D {{{\left| {{k_a}(z)} \right|}^q}d{A_\alpha }(z)} \le \mu (D)^{\frac{{p - q}}{p}}$.
\section{Main Results}
In this section first we define Generalized Carleson measures.
\begin{defn} Let $\mu $ finite positive Borel measure on $\mathbb{D}$ and $p > 0$. We say that $\mu$ is a $(L_a^{\alpha,p}(\mathbb{D}),p)$ -generalized  Carleson measure  on $L_a^{\alpha ,p}(\mathbb{D})$ if there exists a closed subspace  $M\subseteq L_a^{\alpha ,p}(\mathbb{D})$ such that
$$\int\limits_{\mathbb{D}}{\left|f(z)\right|}^{p}d\mu \le C\left\|f\right\|_{{L_a^{\alpha ,p}}}^{p}$$
for all  $f\in M$. Moreover, if we set $M= {L_a^{\alpha ,p}}(\mathcal{A})=E_{\mathcal{A}}(L_a^{\alpha ,p}(\mathbb{D}))$ for some sigma-subalgebra $\mathcal{A}$, then we say that $\mu$ is a $({L_a^{\alpha ,p}},p )$ -conditional Carleson measure  on $L_a^{\alpha ,p}(\mathbb{D})$  if there exists  $C > 0$  such that
$$\int_\mathbb{D}{\left| E_{\mathcal{A}}(f)(z) \right|}^{p}d\mu \le C\left\|f\right\|_{{L_a^{p,\alpha }}}^{p}$$
for all  $f\in L_a^{\alpha ,p}(\mathbb{D})$.
\end{defn}
Now we find an upper bound for the evaluation function $f\rightarrow E_{\mathcal{A}}(f)(z)$.
\begin{thm}\label{17.3} Let $f\in L_a^{p,\alpha }(\mathbb{D})$ such that $E_{\mathcal{A}}(f)\in L_a^{p,\alpha}(\mathcal{A})$, as stated in Theorem \ref{12}, $0< p<\infty$  and $\alpha >-1$. If there exists $r>0$ such that $c_z \subset D(z,r)$, then we have
$$|E_{\mathcal{A}}(f)(z)|\leq{E_{\mathcal{A}}\left|f\right|(z)}\leq C_r\mathop {\sup }\limits_{a \in D_0}  \left| {{k_a}(z)} \right|{\left\| f \right\|_{L_a^{p,\alpha }}}$$
for all $z \in {D_0}$.
\end{thm}
\begin{proof}
It is known that for the conditional expectation $E_{\mathcal{A}}$, we have $|E_{\mathcal{A}}(f)|\leq E_{\mathcal{A}}(|f|)$. Moreover, there exists some $\zeta _n \in {c_z}$ such that for each $ {\zeta _j} \in {c_z}$ we have $|f({\zeta _j})| \le |f({\zeta _n})|$. Hence we get that
 \begin{eqnarray}
\nonumber E_{\mathcal{A}}{\left| {f} \right|(z)} &= &( \sum\limits_{{\zeta _j}\in{c_z}} \omega({\zeta _j})(| {f} |({\zeta _j})))\\
\nonumber & \le & {({\sum\limits_{{\zeta _j}\in{c_z}} \omega  ({\zeta _j}))( {\left| {f({\zeta _n})} \right|} )}}.
 \end{eqnarray}
 Also by the Lemma\ref{3} we have
 \begin{eqnarray}
 \nonumber { {({\sum\limits_{{\zeta _j}\in{c_z}} \omega  ({\zeta _j})} )( {\left| {f({\zeta _n})} \right|} )} }&= &{ {| {f({\zeta _n})}|} }\\
\nonumber &\le & \frac{{\left\| f \right\|_{L_a^{p,\alpha }}} }{{(1 - |{\zeta _n}{|^2})}^{(2 + \alpha )/p}}.
  \end{eqnarray}
Moreover, by Lemmas \ref{4.5} and Lemma\ref{3.5} we conclude that
 \begin{eqnarray}
\nonumber \frac{{\left\| f \right\|_{L_a^{p,\alpha }}} }{{(1 - |{\zeta _n}{|^2})}^{(2 + \alpha )/p}} & \leq &  \frac{{C_r\left\| f \right\|_{L_a^{p,\alpha }}}}{{(1 - {{\left| z \right|}^2})}^{(2 + \alpha )/p}} \\
\nonumber &= & C\mathop {\sup }\limits_{b \in D_0}  {\left| {k_b(z)} \right|}\left\| f \right\|_{L_a^{p,\alpha }}.
 \end{eqnarray}
 for $\beta(z,{\zeta _n})\leq r$. This completes the proof.
 \end{proof}
Here we get the next corollary.

\begin{cor}\label{17.5}  Under the assumptions of Theorem \ref{17.3} we have
$$\left| E_{\mathcal{A}}{(k_a)(z)} \right| \le E_{\mathcal{A}}({\left| {{k_a}(z)} \right|}) \le C_r\mathop {\sup }\limits_{a \in D_0} {\left| {k_a(z)} \right|}\qquad{ z\in D_0}$$
 for all $a\in D_0$.
 \end{cor}
 Let  $0<p<\infty$ and $\alpha >-1$. Then the function ${\Psi}_{a }^{\alpha}(\mu)(z)$  defined as
$${\Psi}_{a }^{\alpha}(\mu)(z)=\int_{\mathbb{D}}{{{\left( \frac{1-{{\left| a  \right|}^{2}}}{{{\left| 1-\overline{a } z  \right|}^{2}}} \right)}^{{(2+\alpha )}}}}d\mu(z),$$
is well defined. In the sequel we provide some equivalence conditions to conditional- Carleson measure on the Bergman spaces.
\begin{thm}\label{18.5} Let $u$ be an analytic function and $\mu$ be a finite Borel measure on $ \mathbb{D}$, $ 0<p\leq q<\infty $ and $\alpha ,\beta >-1$. In addition, let $\mathcal{A}=\mathcal{A}(\varphi)$, in which $\varphi$ is a non-constant analytic self map on $D$. If $\sum\limits_{k = 1}^\infty {\frac{1}{{{{(1 - tanh(r)\left| a_k \right|)}^2}}}}<\infty$ for the sequence $\{a_k\}$ in the Lemma \ref{8} and some $0<r<1$. Then the followings are equivalent:
\begin{enumerate}
\item There exists $ C_1>0$  such that for every $f\in L_a^{p,\alpha}(\mathbb{D})$ we have
\[{\int_{\mathbb{D}} {\left| {{E_{\mathcal{A}}}(f)(z)} \right|} ^p}d\mu (z) \le C_1\left\| f \right\|_{L_a^{p,\alpha }}^p.\]
\item There exists  $ C_2$  such that
$$\mu (D(a,r)) \le {C_2}{\left( {\frac{{1 - {{\left| a \right|}^2}}}{{{{(1 - tanh{\mkern 1mu} r\left| a \right|)}^2}}}} \right)^{{{(\alpha+2)}}}},$$
for all $a\in D$ such that $D(a,r)\in \mathcal{A}$.
\item There exists  $ C_3$ such that
$${\Psi}_{a }^{\alpha}(\mu)(z)\leq C_3,$$
for all $a\in D$ such that $f^{p,\alpha}_{a}$ is $\mathcal{A}$-measurable.
\end{enumerate}
\end{thm}
\begin{proof} {(3) $\longrightarrow $  (1)}
Let (3) holds.  we have
\begin{eqnarray}
\nonumber\int_{{{\mathbb{D}_0}_{{}}}}{|}E(f)(z){{\mid }^{p}}d\mu (z)  &\le & \int\limits_{\mathbb{D}_0}{{{(\left| f({z}) \right|)}^{p}}}d\mu (z)\\
\nonumber  &\leq & \int_{{\mathbb{D}_{0}}} {{{\left( {\frac{1}{{1 - {{\left| {{z}} \right|}^2}}}} \right)}^{{{(\alpha  + 2)}}}}} d\mu (z)\left\| f \right\|_{L_a^{p,\alpha }}^p\\
\nonumber &\approx & \mathop  \int_{{\mathbb{D}_{0}}} {(\frac{{(1 - {{\left| a \right|}^2})}}{{{{\left| {1 - \bar az} \right|}^2}}})}^{{{(\alpha  + 2)}}} d\mu (z)\left\| f \right\|_{L_a^{p,\alpha }}^p \\
\nonumber &= & {\Psi}_{a }^{\alpha}(\mu)(z)\left\| f \right\|_{L_a^{p,\alpha }}^p \\
\nonumber &\le & {C_3}\left\| f \right\|_{L_a^{p,\alpha }}^p.
\end{eqnarray}
$(1)\longrightarrow(3)$
 Let $0<p<\infty$ and $\alpha ,\beta >-1$. Since $ || f^{p,\alpha}_{a}||_{L_{a}^{p,\alpha}}^{p} = 1$, then we have
\begin{eqnarray}
\nonumber\Psi _a ^{\alpha}(\mu)(z) &=&\int_{D} {{{\left( {\frac{{1 - {{\left| a \right|}^2}}}{{{{\left| {1 - \bar az} \right|}^2}}}} \right)}^{(\alpha  + 2)}}} d\mu (z)\\
\nonumber &=&\int_{\mathbb{D}}{\left| {f^{p,\alpha}_{a}} (z)\right|^p}d\mu (z)\\
\nonumber &= &\int_{\mathbb{D}} \left|E(f^{p,\alpha}_{a}(z))\right|^p d\mu(z)\\
\nonumber &\leq & {C_1}.
\end{eqnarray}
(2) $\longrightarrow $ (3)\\
Let $ 0<p<\infty $ and $\alpha  >-1$. First we assume that $a=0 $, hence $\mu({D_0})\le C_2$. If ${|a|\leq\frac{3}{4}} {( a\in D_0)}$, then by \cite{6} we obtain that
 \begin{eqnarray}
 \nonumber\int_{\mathbb{D}_0}{{{\left( \frac{1-{{\left| a  \right|}^{2}}}{{{\left| 1-\overline{a } \omega  \right|}^{2}}} \right)}^{(2+\alpha )}}}d\mu(\omega )&\leq &\mu({D_0})\\
\nonumber &\leq & C_2.
\end{eqnarray}
If $\left| a \right| > \frac{3}{4} {( a\in D_0)}$, then we define $ {E_n}$ as
$${E_n}=\{z\in {D_0}:\left|{z-\frac{a}{{\left| a \right|}}}\right|< {2^n}(1-{\left|a\right|})\} \qquad n = 1,2,3,...$$
such that
\begin{align*}
 \mu ({E_n})&\leq({2^n}(1-{\left| a \right|}))^{(\alpha+2)}\\
&\leq\rm({2^n}(1 - {\left| a \right|^2}))^{(\alpha+2)}.
\end{align*}
Moreover we have
\begin{align*}
 \frac{1-|a|^{2}}{|1-\bar{a}z|^{2}}&\leq\frac{1}{1-|a|}\\
 &\leq \frac{2}{1-|a|^{2}}\qquad\qquad a\in {{E}_{1}},\\
 &\frac{1-|a{{|}^{2}}}{|1-\bar{a}z{{|}^{2}}}\\
 &\leq\frac{1}{{{2}^{2n}}{(1-|a|)}}\\
 &\leq\frac{2}{{{2}^{2n}}(1-|a|^2)}\qquad a\in {{E}_{n}}\setminus{{E}_{n-1}}.
\end{align*}
Since for each $r>0$ we have  $0<tanhr<1$, then we conclude that
\begin{eqnarray}
 \nonumber\int_{D_0}{{{\left( \frac{1-{{\left| a  \right|}^{2}}}{{{\left| 1-\overline{a } \omega  \right|}^{2}}} \right)}^{(2+\alpha )}}}d\mu(\omega ) &\leq &\int_{E_1}{{{\left( \frac{1-{{\left| a  \right|}^{2}}}{{{\left| 1-\overline{a } \omega  \right|}^{2}}} \right)}^{(2+\alpha )}}}d\mu(\omega )\\
 \nonumber &+&\sum\limits_{n=2}^{\infty }\int_{{{E}_{n}}\setminus{{E}_{n-1}}}{{{\left( \frac{1-{{\left| a  \right|}^{2}}}{{{\left| 1-\overline{a } \omega  \right|}^{2}}} \right)}^{(2+\alpha )}}}d\mu(\omega )\\
 \nonumber&\leq&\sum\limits_{n=1}^{\infty }\frac{2}{{{{({2^{2n}}(1 - |a{|^2}))}^{{{(2 + \alpha )}}}}}}{\mu ({E_n})}\\
 \nonumber&\leq&{C_2}\sum\limits_{n = 1}^\infty{\frac{1}{{{2^{n(\alpha  + 2)}}}}} \\
 \nonumber&\leq&{C_3}.
\end{eqnarray}
(3) $\longrightarrow $(2)\\
Suppose  (3) is true then with the Lemma  \ref{3.5} and \ref{3.6}, we have
 \begin{eqnarray}
\nonumber&{\left( {\frac{{{{(1 - tanh(r)\left| a \right|)}^2}}}{{1 - {{\left| a \right|}^2}}}} \right)^{{(\alpha  + 2)}}}\mu (D(a,r))\\
\nonumber&\leq&\int_{D(a,r)} {{{\left( {\frac{{{{(1 - tanh(r)\left| a \right|)}^2}}}{{1 - {{\left| a \right|}^2}}}} \right)}^{(\alpha  + 2)}}d\mu (z)} \\
\nonumber&\leq&\int_{D(a,r)} {{{\left( {\frac{{{{(1 + tanh(r)\left| a \right|)}^2}}}{{1 - {{\left| a \right|}^2}}}} \right)}^{(\alpha  + 2)}}d\mu (z)} \\
\nonumber&=&{\int_{D(a,r)} \mathop {\sup }\limits_{z \in D(a,r)}\mathop  {\left| {f^{p,\alpha}_{a}} (z)\right|^p}d\mu (z)} \\
\nonumber&\leq&{\int_{D_0} \mathop {\sup }\limits_{z \in D_0}\mathop  {\left| {f^{p,\alpha}_{a}} (z)\right|^p}d\mu (z)} \\
\nonumber&=&\int_{\mathbb{D}_0} {\left( {\frac{1}{{1 - {{\left| z \right|}^2}}}} \right)^{(2 + \alpha )}}d\mu (z) \\
\nonumber&\approx &\int_{\mathbb{D}_0} {{\left( {\frac{{1 - {{\left| a \right|}^2}}}{{{{\left| {1 - \bar az} \right|}^2}}}} \right)}^{(2 + \alpha )}d\mu } (z) \\
\nonumber&=&\Psi_a ^{\alpha}(\mu )(z) \\
 \nonumber&\leq&C_2.
\end{eqnarray}
 For all  $a\in{\mathbb{D}_0}$.\\
 (2) $\longrightarrow $  (1)\\
   By some properties of conditional expectation, we have
\begin{eqnarray}
\nonumber \int_{\mathbb{D}_0} {{\left| {E(f)(z)} \right|}^p}d\mu  &\le &   \int_{\mathbb{D}_0} {{| {f(z)} |}^p}d\mu \\
\nonumber &\le &\mathop\sum\limits_{k = 1}^\infty  {\mu (D({a_k},r))} \mathop {\sup }\limits_{z \in D({a_k},r)} |f(z)|^{p} .
\end{eqnarray}
Here by the Lemma  \ref{15.5} ,[\cite{5} page 60] we have
$$\sup \{ {\left| {f(z)} \right|^p}:z \in D({a_k},r)\}  \le \frac{C}{{{{(1 - {{\left| {{a_k}} \right|}^2})}^{\alpha  + 2}}}}{\int_{D({a_{k,}}r)} {\left| {f(w)} \right|} ^p}d{A_\alpha }(w).$$
using Holder’s inequality and let $q$ be the conjugate index of ${q'}$,then
\begin{eqnarray}
\nonumber &&\sum\limits_{k = 1}^\infty  {\mu (D({a_k},r))}\mathop {\sup }\limits_{z \in D({a_k},r)}  |f(z)|^{p}\\
\nonumber  &\le &  \mathop \sum \limits_{k = 1}^\infty  \frac{{\mu (D({a_k},r))}}{{(1-|a_k|^2)}^{\alpha+2}}\int_{D({a_k},r)} { |f(w)|^{p}d{A_\alpha }(w)} \\
\nonumber & \le & {\left( {\sum\limits_{k = 1}^\infty  {{{\left( {\frac{{\mu (D({a_k},r))}}{{{{(1 - |{a_k}{|^2})}^{\alpha  + 2}}}}} \right)}^{{\textstyle{1 \over q}}}}} } \right)^q}{\left( {\sum\limits_{k = 1}^\infty  {{{\left( {\int_{D({a_k},r)} {|f(w){|^p}d{A_\alpha }(w)} } \right)}^{{\textstyle{1 \over q'}}}}} } \right)^{q'}} \\
\nonumber &\le &{\left( {\sum\limits_{k = 1}^\infty  {{{\left( {\frac{1}{{{{(1 - tanh(r)\left| a_k \right|)}^2}}}} \right)}^{{\textstyle{{\alpha  + 2} \over q}}}}} } \right)^q}\\
\nonumber &\times &{\left( {\sum\limits_{k = 1}^n {{{\left( {\int_{D({a_k},2r)} {|f(w){|^p}d{A_\alpha }(w)} } \right)}^{{\textstyle{1 \over q'}}}}} } \right)^{q'}}\\
\nonumber &\le &{N}\left\| f \right\|_{L_a^{p,\alpha }}^p.
\end{eqnarray}
Then
\begin{eqnarray}
\nonumber \int_{\mathbb{D}_0} {{\left| {E(f)(z)} \right|}^p}d\mu  &\le &C_1 \left\| f \right\|_{L_a^{p,\alpha }}^p.
\end{eqnarray}
\end{proof}
\textbf{Notice} In Theorem \ref{18.5} the notation $A \approx B$ means that there exists a positive constant $C$ independent of $A$ and $B$ such that $C^{-1}B \leq A \leq CB$. Moreover, the best constants $C_1$, $C_2$ and $C_3$ are in fact comparable, i,e., there exists a positive constant
$M$ such that
$$\frac{1}{M}{{\rm{C}}_1}{\rm{ }} \le {\rm{ }}{{\rm{C}}_2}{\rm{ }} \le {\rm{ M}}{{\rm{C}}_1},$$$$\frac{1}{M}{{\rm{C}}_1}{\rm{ }} \le {{\rm{C}}_3}{\rm{ }} \le {\rm{ M}}{{\rm{C}}_1}.$$
 \begin{thm}\label{0.5} Under the assumptions of the Theorem \ref{18.5} we obtain that the weighted Conditional expectation operator $M_{u}E$ is bounded from $ L_{a }^{p,\alpha }$ into $L_{a }^{p,\beta }$ if and only if there exists some C such that
\[{\Psi}_{a }^{\alpha}({{\mu }_{u}^\beta})(z)\leq C\]
for $a,z \in D_0=\{ z \in {\rm{D}}:\varphi '(z) \ne 0\}$.
\end{thm}
\begin{proof}
Suppose that ${{M}_{u}}E$ is bounded. So we can find $C>0$ such that
\[\Vert{M_u}E(f)\Vert_{L_a^{p,\beta }}^p \le C\left\| {f} \right\|_{L_a^{p,\alpha }}^p,\]
and
\begin{eqnarray}
 \nonumber\Vert{{M}_{u}}E(f)\Vert_{L_a^{p,\beta}}^{p}&=&\int_{{\mathbb{D}_{{}}}}{|}E(f)(z){{\mid }^{p}}{{{\left| u(z) \right|}}^{p}}d{{A}_{\beta }}(z)\\
  \nonumber&=&\int_{{\mathbb{D}_{{}}}}{|}E(f)(z){{\mid }^{p}}d{{\mu }_{u}^\beta}(z)\\
 \nonumber    &\le  & C\left\| f{{(z)}^{{}}} \right\|_{L_a^{p,\alpha}}^{p},\\
\end{eqnarray}
  In which ${{\mu }_{u}^\beta}={{\int_\mathbb{D}{\left| u(z) \right|}}^{p}}d{{A}_{\beta }}(z)$. This means that ${d\mu }_{u}^\beta$ is
an $(L_a^{p,\alpha})$-expectation Carleson measure. By  theorem\ref{18.5}, this is equivalent to
\[\int_{\mathbb{D}_0}{{{\left( \frac{1-{{\left| a  \right|}^{2}}}{{{\left| 1-\overline{a } z  \right|}^{2}}} \right)}^{(2+\alpha )}}}{\left| u(z) \right|}^{p}d{A}_\beta (z)\leq{\Psi}_{a }^{\alpha}(\mu_u^\beta)(z)\leq C.\]
This completes the proof.
\end{proof}

\begin{exam} Let $u$ be an analytic function on $ \mathbb{D} $, $ 0<p\le q<\infty $ and $\alpha ,\beta >-1$. If the conditional expectation operator $E$ is identity, then the Multiplication operator $ {{M}_{u}}$ is bounded from $ L_{a }^{p,\alpha }$ into $L_{a }^{q,\beta }$ if and only if there exists some $ C$ such thet
\[\int_D {{{\left( {\frac{{1 - {{\left| a \right|}^2}}}{{{{\left| {1 - \bar az} \right|}^2}}}} \right)}^{\frac{{(2 + \alpha )q}}{p}}}} {\left| {u(z)} \right|^q}d{A_\beta }(z) \le C.\]
    \end{exam}



\begin{thebibliography}{99}
\bibitem{al} A. B. Aleksandrov, Measurable partitions of the circumf erence, induced by inner functions, Translated from Zepiski Nauchnykh Seminarov Leningradskogo otdeleniya Matematicheskogo Instituta im. V. A. Steklova AN SSSR, Vol 149, (1986) 103-106.
\bibitem{jb} J. A. Ball, Hardy space expectation operators and reducing subspace, Proc. Amer. Math. Soc. 47 (1975) 351-357.
\bibitem{2}Hastings, W.W., A Carleson measure theorem for Bergman spaces. Proceedings of the American Mathematical Society, 1975. 52(1): p. 237-241.
\bibitem{3}Luecking, D.H., Forward and reverse Carleson inequalities for functions in Bergman spaces and their derivatives. American Journal of Mathematics, 1985. 107(1): p. 85-111.
\bibitem{jm}Jabbarzadeh,M.R.and M.Moradi,C-E type Toeplitz operators on $L_a^2(\mathbb{D})$ . Operators and Matrices, 2017. 11: p. 875-884.
\bibitem{5}Vukotić, D., A sharp estimate for Proceedings of the American Mathematical Society, 1993. 117(3): p. 753-756.
\bibitem{6}Aulaskari, R., D.A. Stegenga, and J. Xiao, Some subclasses of BMOA and their characterization in terms of             Carleson measures. Rocky Mountain Journal of Mathematics, 1996. 26: p. 485-506.
\bibitem{hjj}Hornor, W.E. and J.E. Jamison, Properties of lsometry-Inducing Maps of the Unit Disc. Complex Variables and Elliptic Equations, 1999. 3: p. 69-84.
\bibitem{cs}Carswell, B.J. and M.I. Stessin, Conditional expectation and the Bergman projection. Journal of Mathematical Analysis and Applications, 2008. 341(1): p. 270-275.
\bibitem{11}Zhu, K., Spaces of holomorphic functions in the unit ball. Vol. 226. 2005: Springer Science .
\bibitem{12}Cucković, Z. and R. Zhao, Weighted composition operators between different weighted   Bergman spaces and different Hardy spaces. Illinois Journal of mathematics, 2007. 51(2): p. 479-498.
\bibitem{14}Zhu, K., Operator theory in function spaces. 2007: American Mathematical Soc.
\bibitem{jh}Jabbarzadeh, M. and M. Hassanloo, Conditional expectation operators on the Bergman spaces. Journal of Mathematical Analysis and Applications, 2012. 385(1): p. 322-325.
\bibitem{27}Ueki, S.-i., Order bounded weighted composition operators mapping into the Bergman space. Complex Analysis and Operator Theory, 2012. 6(3): p. 549-560.
\end{thebibliography}
\end{document}